\documentclass[11pt]{article}

\usepackage{amsthm}
\usepackage{amssymb}
\usepackage{amsmath}
\usepackage{pinlabel}
\usepackage{graphicx}

\theoremstyle{plain}
\newtheorem{tetel}{Theorem}[section]
\newtheorem{all}[tetel]{Proposition}
\newtheorem{lemma}[tetel]{Lemma}
\newtheorem{kov}[tetel]{Corollary}

\theoremstyle{definition}
\theoremstyle{remark}\newtheorem{megj}[tetel]{Remark}
\newtheorem{pelda}[tetel]{Example}

\begin{document}

\title{Meridian twisting of closed braids and the Homfly polynomial}
\author{Tam\'as K\'alm\'an\\ University of Tokyo}
\maketitle

\begin{abstract}
Let $\beta$ be a braid on $n$ strands, with exponent sum $w$. Let $\Delta$ be the Garside half-twist braid. We prove that the coefficient of $v^{w-n+1}$ in the Homfly polynomial of the closure of $\beta$ agrees with $(-1)^{n-1}$ times the coefficient of $v^{w+n^2-1}$ in the Homfly polynomial of the closure of $\beta\Delta^2$. This coincidence implies that the lower Morton--Franks-Williams estimate for the $v$--degree of the Homfly polynomial of $\widehat\beta$ is sharp if and only if the upper MFW estimate is sharp for the $v$--degree of the Homfly polynomial of $\widehat{\beta\Delta^2}$. 
\end{abstract}

In this note, an old story is told again in a way that yields a surprising new result (Figure \ref{fig:newton38} shows a representative example). Indeed, the only parts that could not have been written twenty years ago are some speculation about Khovanov--Rozansky homology and Remark \ref{rem:dan}, which do not belong to the proof, and the proof of Proposition \ref{murakami}. The latter is a well known fact and in particular it too can be established using skein theory. The new proof is included for completeness and because it is short.

The main theorem will be stated in the introduction and its proof occupies the entire section \ref{sec:proof}. We conclude the paper with examples and remarks. 




\section{Introduction}




Braids 
will be drawn horizontally. In particular, the standard generators $\sigma_i$ of the braid group $B_n$ \cite{birman} appear as a crossing $\includegraphics[viewport=2 2 20 10]{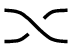}$ 
sandwiched between groups of $i-1$ and $n-i-1$ trivial strands, respectively. $1_n$ is the trivial braid on $n$ strands. The usual closure of the braid $\beta$ will be denoted by $\widehat\beta$. It gets its orientation from orienting the strands of $\beta$ from left to right.

The \emph{framed Homfly polynomial} $H_D$ is a two-variable, integer-coefficient Laurent polynomial in the indeterminates $v$ and $z$. It is associated to any 
oriented link diagram $D$, and it is uniquely defined by the skein relations
\[H_{\includegraphics[totalheight=8pt]
{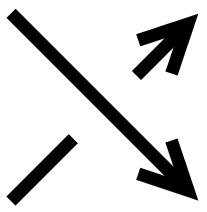}}
-H_{\includegraphics[totalheight=8pt]{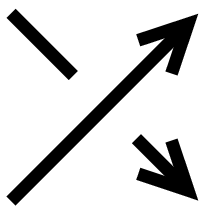}}
=zH_{\includegraphics[totalheight=8pt]
{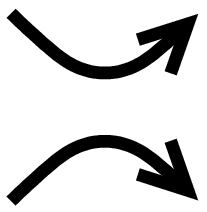}};\quad
H_{\includegraphics[totalheight=8pt]{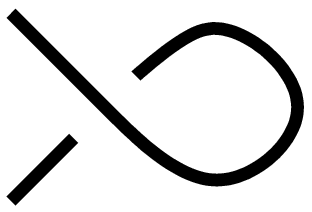}}
=vH_{\includegraphics[totalheight=8pt]{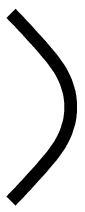}};\quad
H_{\includegraphics[totalheight=8pt]{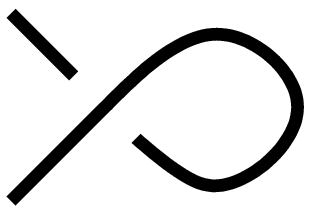}}
=v^{-1}H_{\includegraphics[totalheight=8pt]{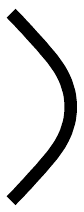}}\]
and the requirement (normalization) that for the crossingless diagram of the unknot (with either orientation), $H_{\includegraphics[totalheight=8pt]{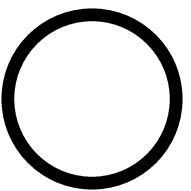}}(v,z)=1$. 
$H$ is invariant under regular isotopy of knot diagrams. The Homfly polynomial itself is \[P_D(v,z)=v^wH_D(v,z),\] where $w$ is the writhe of $D$. Unlike $H$, it is an oriented link invariant: $P(v,z)$ takes the same value for 
diagrams that represent isotopic oriented links.

In this paper, we record Homfly coefficients 
on the $vz$--plane. (We hope that our blurring of the distinction between 
the indeterminates and their exponents will not lead to confusion.) We will often speak about \emph{columns} of the Homfly polynomial. These can be equivalently thought of as polynomials in $z$ that appear as the coefficients of various powers of $v$.

\begin{pelda}\label{ex:newt}
The Homfly polynomial of the torus knot $T(3,5)$ is
\begin{multline*}
P_{T(3,5)}(v,z)=z^8v^8+8z^6v^8-z^6v^{10}+21z^4v^8-7z^4v^{10}\\
+21z^2v^8-14z^2v^{10}+z^2v^{12}+7v^8-8v^{10}+2v^{12}\\
=(z^8+8z^6+21z^4+21z^2+7)v^8-(z^6+7z^4+14z^2+8)v^{10}+(z^2+2)v^{12},
\end{multline*}
and we will write it as shown in Figure \ref{fig:newton35}.
\begin{figure}[h]
\labellist
\small
\pinlabel $z$ at -5 385 
\pinlabel $z=8$ at -55 320
\pinlabel $z=0$ at -55 70
\pinlabel $v=8$ at 370 -5
\pinlabel $v=12$ at 560 -5
\pinlabel $v$ at 670 55
\endlabellist
   \centering
   \includegraphics[width=.5\linewidth]{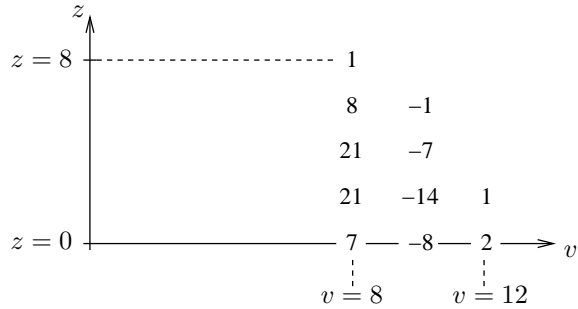}
   \caption{The Homfly polynomial of the $(3,5)$ torus knot}
   \label{fig:newton35}
\end{figure}                              
\end{pelda}

Next, recall the famous Morton--Franks-Williams inequality \cite{fw,morton} which (put somewhat sloppily) says that
\begin{equation}\label{ineq:mfw}
\text{braid index}\ge\text{number of non-zero columns in }P
\end{equation}
for any oriented link type. (For knot theoretical definitions such as braid index, we refer the reader to \cite{crom}.)
In its standard proof, this is derived from the following pair of inequalities. Let $\beta$ be a braid word on 
$n$ strands, with exponent sum $w$.
Then,
\begin{equation}\label{lower}
w-n+1\le\text{lowest }v\text{--degree of }P_{\widehat\beta}
\end{equation}
and
\begin{equation}\label{upper}
\text{highest }v\text{--degree of }P_{\widehat\beta}\le w+n-1.
\end{equation}

We will refer to \eqref{lower} as the \emph{lower MFW estimate} and to \eqref{upper} as the \emph{upper MFW estimate}. Similarly, we might call $w-n+1$ the \emph{lower MFW bound} (for $\beta$) and $w+n-1$ the \emph{upper MFW bound}. Graphically, these inequalities mean that the left column of the Homfly polynomial $P_{\widehat\beta}$ is to the right of $v=w-n+1$ while the right column is to the left of $v=w+n-1$.


\begin{pelda}\label{ex:35}
If we represent $T(3,5)$ with the braid word 
\[\beta=(\sigma_1\sigma_2)^5=\,\includegraphics[width=.17\linewidth,viewport=2 2 37 10]{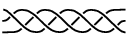},\] then 
$n=3$,
$w=10$,
$w-n+1=8$,
$w+n-1=12$,
and we see that both the lower and the upper MFW estimates are sharp for this braid. Indeed, the Homfly polynomial of Example \ref{ex:newt} has $3$ columns.
\end{pelda}

We will denote the Garside braid (positive half twist) on $n$ strands by $\Delta_n$ or simply by $\Delta$. Then $\Delta^2$ represents a positive full twist. (For example, $\Delta_3=\,\includegraphics[width=.05\linewidth,viewport=2 2 15 10]{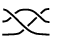}$ and $\Delta_3^2=\,\includegraphics[width=.1\linewidth,viewport=2 2 28 10]{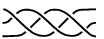}$.) Since $\Delta^2$ contains $n(n-1)$ positive crossings, the braid $\beta\Delta^2$ still has $n$ strands 
but exponent sum $w+n(n-1)$. Thus, 
\begin{equation}\label{newupper}
\text{the upper MFW bound for }\beta\Delta^2\text{ is }w+n(n-1)+n-1=w+n^2-1.
\end{equation}

Now we are ready to state our main theorem.

\begin{tetel}\label{thm:jobbbal}
For any braid $\beta$ on $n$ strands, the lower MFW estimate is sharp if and only if the upper MFW estimate is sharp for the braid $\beta\Delta^2$. If this is the case, then 
\begin{equation}\label{eq:jobbbal}
\text{left column of }P_{\widehat\beta}=(-1)^{n-1}\text{ right column of }P_{\widehat{\beta\Delta^2}}.
\end{equation}
\end{tetel}

\begin{pelda}\label{ex:torus}
If we add a full twist to the braid of Example \ref{ex:35}, we observe a change in the Homfly polynomial as shown in Figure \ref{fig:newton38} (cf.\ Corollary \ref{cor:poz}).
\begin{figure}[h]
\labellist
\small
\pinlabel $z$ at 155 910
\pinlabel $z$ at 155 350
\pinlabel $\beta=$ at 70 730
\pinlabel $\beta\Delta^2=$ at 95 175
\pinlabel $v$ at 1105 20
\pinlabel $v$ at 1105 580
\endlabellist
   \centering
   \includegraphics[width=.7\linewidth]{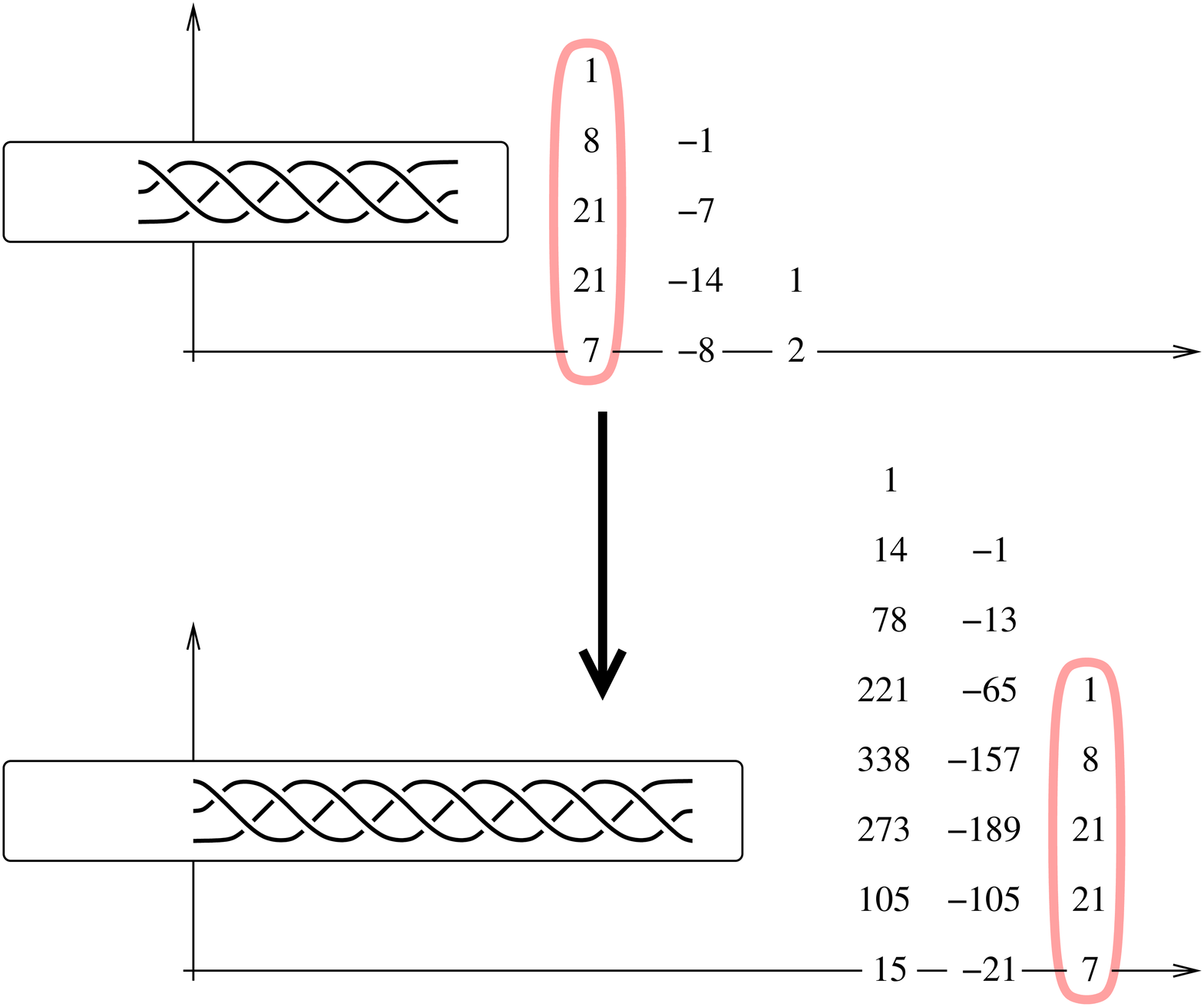}
   \caption{The effect of adding a full twist on the Homfly polynomial}
   \label{fig:newton38}
\end{figure}
\end{pelda}

Somewhat more precisely, we can claim the following.

\begin{tetel}\label{thm:preciz}
For any braid $\beta$ of $n$ strands and exponent sum $w$,
\begin{multline}\label{eq:claim}
\text{the coefficient of }v^{w-n+1}\text{ in }P_{\widehat\beta}\\
=(-1)^{n-1}\text{ the coefficient of }v^{w+n^2-1}\text{ in }P_{\widehat{\beta\Delta^2}}.\end{multline}
\end{tetel}

From this latter statement, Theorem \ref{thm:jobbbal} follows immediately. 
Indeed, comparing Theorem \ref{thm:preciz} with \eqref{lower}, \eqref{upper}, and \eqref{newupper}, we see that \eqref{eq:claim} either says that $0=0$, or the more meaningful formula \eqref{eq:jobbbal}, depending on whether the sharpness condition is met.

We give yet another formulation of our result which fits best with the method of our proof. It is important to stress that whenever the framed Homfly polynomial of a closed braid is taken, $\widehat\beta$ means the diagram that is the union of a
diagram of $\beta$ and $n$ disjoint \emph{simple} curves \emph{above} it to form its closure.

\begin{all}\label{pro:framed}
For any braid $\beta$ of $n$ strands and exponent sum $w$,
\begin{multline}\label{eq:framed}
\text{the coefficient of }v^{-n+1}\text{ in }H_{\widehat\beta}\\
=(-1)^{n-1}\text{ the coefficient of }v^{n-1}\text{ in }H_{\widehat{\beta\Delta^2}}.
\end{multline}
\end{all}

It is easy to see that \eqref{eq:claim} and \eqref{eq:framed} are equivalent. The left hand side of \eqref{eq:framed} agrees with the coefficient of $v^wv^{-n+1}$ in $v^wH_{\widehat\beta}=P_{\widehat\beta}$, which is the left hand side of \eqref{eq:claim}. Similarly, the right hand side of \eqref{eq:framed} equals (plus or minus) the coefficient of $v^{w+n(n-1)}v^{n-1}$ in $v^{w+n(n-1)}H_{\widehat{\beta\Delta^2}}=P_{\widehat{\beta\Delta^2}}$, i.e.\ the right hand sides agree, too.

By reading our formulas from right to left, we get versions of our claims for the case of a full negative twist. (Equally trivial proofs can be derived from the fact that the Homfly polynomials of an oriented link and its mirror image are related by the change of variable $v\mapsto-v^{-1}$.) In particular, corresponding to Theorem \ref{thm:jobbbal}, we have

\begin{kov}
For any braid $\beta$ on $n$ strands, the upper MFW estimate is sharp if and only if the lower MFW estimate is sharp for the braid $\beta\Delta^{-2}$. If this is the case, then 
the right column of $P_{\widehat\beta}$ coincides with $(-1)^{n-1}$ times the left column of $P_{\widehat{\beta\Delta^{-2}}}$.
\end{kov}

The two equivalent sharpness conditions of Theorem \ref{thm:jobbbal} both hold for positive braids \cite{fw}. 
Thus we have

\begin{kov}\label{cor:poz}
For any positive braid $\beta$ on $n$ strands, the left column of its Homfly polynomial agrees with $(-1)^{n-1}$ times the right column of the Homfly polynomial of $\beta\Delta^2$.
\end{kov}

Note however that the Morton--Franks-Williams inequalities are sharp for many non-braid-positive knots, too. Up to $10$ crossings, there are only five knots that do \emph{not} possess braid representations with a sharp (lower) MFW estimate \cite{jones}. Thus, \eqref{eq:claim} is informative for many non-positive braids as well (cf.\ Example \ref{ex:12n187}).

\begin{megj}
For all three of the inequalities, \eqref{ineq:mfw}, \eqref{lower}, and \eqref{upper}, that we quoted from \cite{fw} and \cite{morton}, it makes sense to say that they are \emph{sharp for a braid} (like we did throughout most of this introduction) and also to discuss (as in the previous paragraph) whether they are \emph{sharp for a link}. The latter means that a given oriented link type has a braid representative that turns the inequality into an equality.
\end{megj}

We are unaware of a version of the `twist phenomenon' for the Kauffman polynomial. Likewise, there do not seem to be noteworthy consequences of Theorem \ref{thm:preciz} for the various other knot polynomials that are derived from the Homfly polynomial. (Let us mention though Yokota's work \cite{yok} concerning the Jones polynomial here.) However we do anticipate there to be a twisting formula for the generalization (categorification) of the Homfly polynomial known as Khovanov--Rozansky homology \cite{kr1} (cf.\ Example \ref{ex:12n187}).

\section{Proof of the main theorem}\label{sec:proof}

We will give a detailed proof of Proposition \ref{pro:framed}. It has already been explained how our other claims follow from it. Despite the linear nature of the formulas \eqref{eq:claim} and \eqref{eq:framed}, we will be able to avoid using the Hecke algebra and give a proof based on skein theory.

\subsection{Review of known techniques}

First off, note that for the framed Homfly polynomial the MFW estimates take the form
\begin{equation}\label{ineq:framed}
-n+1\le\text{ any exponent of }v\text{ in }H_{\widehat\beta}(v,z)\le n-1.
\end{equation}

We will need the following fact about the link $\widehat{\Delta_n^2}$ which is formed by $n$ fibers of the Hopf fibration. It is a tiny portion of a known formula. 

\begin{all}\label{murakami}
The highest $v$-exponent in $H_{\widehat{\Delta_n^2}}(v,z)$ is $n-1$ and the single term in which it occurs is $(-1)^{n-1}z^{1-n}v^{n-1}$.
\end{all}

\begin{proof}
This statement can be proven by induction based on the formula $\Delta_n^2=(\sigma_{n-1}\ldots\sigma_1)(\sigma_1\ldots\sigma_{n-1})\Delta_{n-1}^2$, skein relations, inequality \eqref{upper}, and repeated use of arguments similar to Lemma \ref{lem:kiteker}. It can also be inferred from \cite[section 9]{jones}. We will present a third proof, based on a different set of notions. All relevant definitions can be found in \cite{dan}.

\begin{figure}
   \centering
   \includegraphics[width=.2\linewidth]{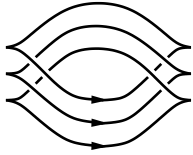}
   \caption{A Legendrian representative of the link $\widehat{\Delta^{-2}}$}
   \label{fig:delta-2}
\end{figure}

Figure \ref{fig:delta-2} shows the front diagram of a Legendrian representative of the link $\widehat{\Delta_n^{-2}}$. The number of its left cusps is $c=n$ and it has Thurston--Bennequin number $-n^2$. As all crossings are negative, the only oriented ruling of this diagram is the one without switches. Rutherford associates to it the number $j=\text{number of switches}-c+1=-n+1$.

Now \cite[Theorem 4.3]{dan} says that the so-called oriented ruling polynomial, in this case $\sum z^j=z^{-n+1}$, agrees with the coefficient of $a^{c-1}=v^{1-c}=v^{1-n}$ in the framed Homfly polynomial of the diagram obtained by smoothing the cusps. Furthermore, the Homfly estimate on the Thurston--Bennequin number \cite{ft} implies that this ($1-n$) is the lowest exponent of $v$ that occurs in $H_{\widehat{\Delta^{-2}}}(v,z)$. Finally, to obtain information on $\widehat{\Delta^2}$, we substitute $-v^{-1}$ for $v$. We find that the highest exponent of $v$ in $H_{\widehat{\Delta^2}}(v,z)$ is indeed $n-1$ and that it occurs with the coefficient $(-1)^{1-n}z^{-n+1}$.
\end{proof}

Note that since the (framed) Homfly polynomial of the $n$-component unlink is $(\frac{v^{-1}-v}z)^{n-1}$, which has left column (coefficient of $v^{1-n}$) $z^{1-n}$, Proposition \ref{murakami} already confirms Proposition \ref{pro:framed} for the special case $\beta=1_n$.

Next, recall that for any braid, a \emph{computation tree} can be built  
using the four types of moves discussed below \cite{fw}. 
Each vertex of the tree carries a \emph{label}, which either means a braid (word) or some polynomial associated to the (closure of) the braid. At first, labels will be interpreted as braids. The moves are as follows.

\begin{itemize}
\item Isotopy (braid group relations)
\item Conjugation: $\beta_1\beta_2\mapsto\beta_2\beta_1$
\item Positive Markov destabilization: $\alpha\sigma_i\in B_{i+1}$ becomes $\alpha\in B_i$
\item Two types of Conway splits, as shown in Figure \ref{fig:conway}.
\begin{figure}[h]
\labellist
\small
\pinlabel $1$ at 135 165
\pinlabel $z$ at 660 160
\pinlabel $-z$ at 135 50
\pinlabel $1$ at 660 45
\endlabellist
   \centering
   \includegraphics[width=.7\linewidth]{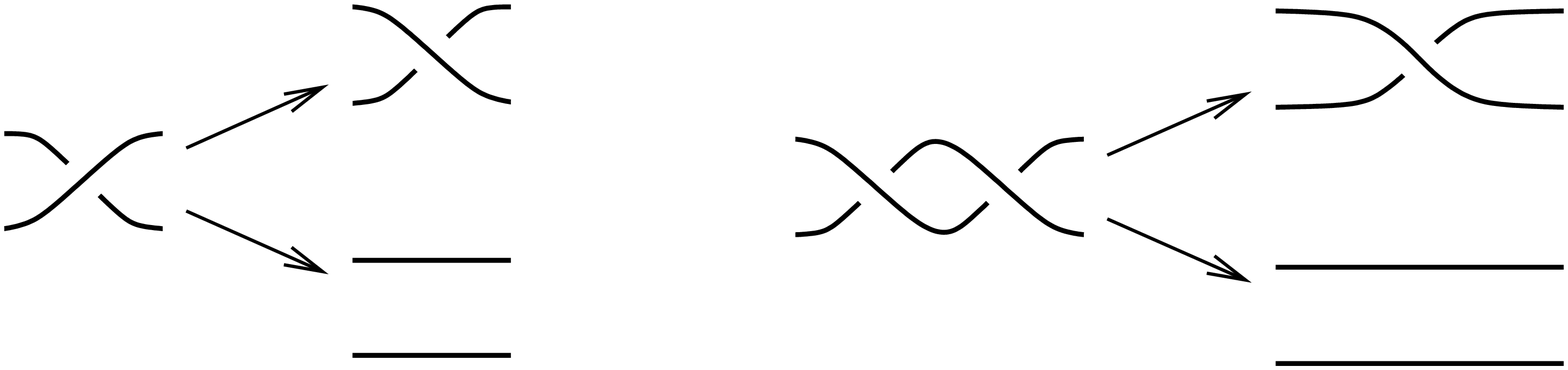}
   \caption{Conway splits}
   \label{fig:conway}
\end{figure}
\end{itemize}

The tree is rooted and oriented. The \emph{terminal nodes} of the tree are labeled with unlinks on various numbers of strands. To be precise, we should note that isotopy and conjugation take place `within the vertices' of the tree; that is to say, by labels we really mean certain equivalence classes of braids. This does not lead to ambiguity because these two operations on braids correspond to regular isotopies of the closures. A Markov destabilization also has a very controlled effect, but for us it will be convenient to treat it already as an edge of the tree (the tail of such an edge is a $2$-valent vertex). Clearly, Conway splits are mostly responsible for the structure of the graph.  

A computation tree can be used to evaluate various knot polynomials. In particular, to obtain the \emph{framed} Homfly polynomial $H$ of the closure of the (label of the) root, we proceed as follows. Label edges resulting from a Conway split as already shown in Figure \ref{fig:conway}. Any edge representing a (positive) Markov destabilization carries the label $v^{-1}$ (see Figure \ref{fig:markov}).  
If a terminal node $x$ is labeled with $1_k$, then also label it with the polynomial $(\frac{v^{-1}-v}z)^{k-1}$. Then, this terminal node contributes
\begin{equation}\label{eq:contrib}
\left(\frac{v^{-1}-v}z\right)^{k-1} h(x),
\end{equation}
where $h(x)$ is the product of the edge labels that appear along the path that connects $x$ to the root. Finally, the framed Homfly polynomial of the root is the sum of these contributions.

It is not essential to insist on trivial braids as terminal nodes. The point is that we must know what the framed Homfly polynomials of the labels of the terminal nodes are. Also, note that the subtree generated by any vertex of a computation tree is a computation tree for that vertex.

\subsection{Proof of Proposition \ref{pro:framed}}

Starting from a computation tree $\Gamma$ for the braid $\beta$, we build a computation tree $\widetilde\Gamma$ for $\beta\Delta^2$. Of course, we will try to imitate $\Gamma$ as much as possible. In particular, we will attempt to delay altering the full twist in the braid word until after the crossings of $\beta$ are all removed. (To see to what extent this is possible, we will have to analyze the four basic moves.) Doing so, we will be able to realize $\Gamma$ as a subtree of $\widetilde\Gamma$ and to read off our result.

An important difference between $\Gamma$ and $\widetilde\Gamma$ is that whenever a Markov destabilization in $\Gamma$ reduces the number of strands, 
the corresponding vertices of $\widetilde\Gamma$ will still be labeled with braids on $n$ strands. (This is simply because as long as $\Delta_n^2$ is in the braid word, we need $n$ strands.) We may think that the reduced strands live on as trivial ``ghost strands'' shown as dotted lines in Figure \ref{fig:markov}.

Two of the four basic steps, namely isotopy and Conway splitting, are completely local and thus they can be carried out unchanged when $\Delta^2$ is attached to the end of $\beta$. Conjugation is easy too, since $\Delta^2$ belongs to the center of $B_n$. Thus, the conjugation move
\[\beta_1\beta_2\mapsto\beta_2\beta_1\quad\text{in }\Gamma\]
can be replaced by an isotopy followed by a conjugation
\[\beta_1\beta_2\Delta^2\mapsto\beta_1\Delta^2\beta_2\mapsto\beta_2\beta_1\Delta^2\quad\text{in }\widetilde\Gamma.\]

Markov destabilization however requires more work. If a step $\alpha\sigma_i\mapsto\alpha$, as shown in Figure \ref{fig:markov}, is part of the computation tree $\Gamma$, then we will realize it in $\widetilde\Gamma$ as part of a Conway split; see Figure \ref{fig:kiteker}. Of course this gives rise to a new, `unnecessary' subtree of $\widetilde\Gamma$ that does not have a counterpart in $\Gamma$.

\begin{figure}
\labellist
\small
\pinlabel $\alpha$ at 75 110
\pinlabel $\alpha$ at 420 110
\pinlabel $v^{-1}$ at 280 115
\endlabellist
   \centering
   \includegraphics[width=.3\linewidth]{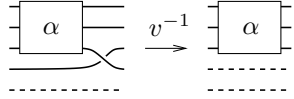}
   \caption{Positive Markov destabilization changes the framed Homfly polynomial $H$ by a factor of $v$.}
   \label{fig:markov}
\end{figure}

\begin{figure}
\labellist
\small
\pinlabel $\alpha$ at 75 495
\pinlabel $\alpha$ at 525 495
\pinlabel $\alpha$ at 1375 495
\pinlabel $\alpha$ at 145 135
\pinlabel $\alpha$ at 960 135
\pinlabel $z$ at 1225 490
\pinlabel $\Delta^2$ at 265 465
\pinlabel $\Delta^2$ at 1530 465
\pinlabel $1$ at 550 325
\pinlabel $v^{-1}$ at 820 140
\endlabellist
   \centering
   \includegraphics[width=\linewidth]{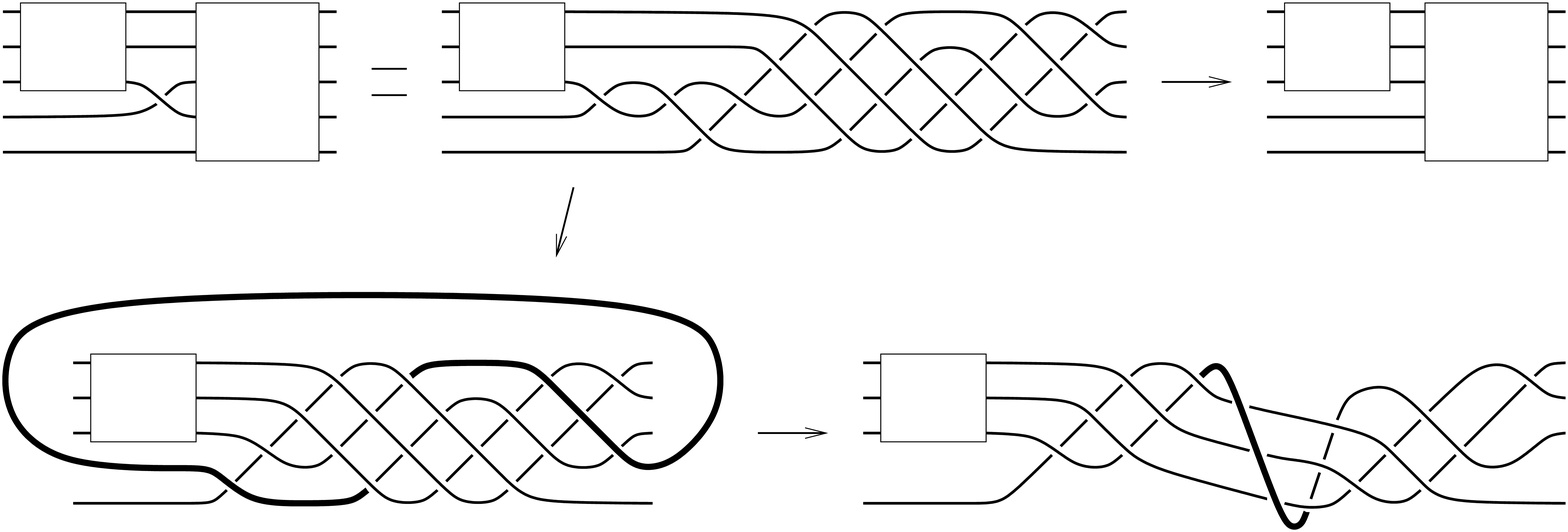}
   \caption{The Conway split (and consequent destabilization) in $\widetilde\Gamma$ that replaces the Markov destabilization in $\Gamma$}
   \label{fig:kiteker}
\end{figure}

A precise description is as follows. $\Delta$ can be represented by positive braid words that either start or end with arbitrarily chosen braid group generators \cite[Lemma 2.4.1]{birman}. Thus we may write $\Delta^2=\sigma_i\delta\sigma_i$ and hence $\alpha\sigma_i\Delta^2=\alpha\sigma_i^2\delta\sigma_i$. We perform the Conway split at $\sigma_i^2$. One of the two resulting braids is $\alpha\sigma_i\delta\sigma_i=\alpha\Delta^2$, as desired. Note that contrary to $\Gamma$, in $\widetilde\Gamma$ the number of strands did not decrease.

The other edge that results from the Conway split points to the `side-product' word $\alpha\delta\sigma_i$. 

\begin{lemma}\label{lem:kiteker}
The closure of the braid $\alpha\delta\sigma_i$ is also the closure of a braid on $n-1$ strands. 
\end{lemma}

\begin{proof}
We explain the move pictured at the bottom of Figure \ref{fig:kiteker}: the thick part of the left hand side diagram is pulled tight to appear as on the right. The braid $\delta\sigma_i$ is almost a full twist. Its strands can be bundled together into three groups: (i) the $i-1$ strands whose endpoints are on the top; (ii) the $2$ strands right below them; (iii) the remaining $n-i-1$ strands. In other words, we think of $\delta\sigma_i$ as the union of three smaller braids. (In Figure \ref{fig:kiteker}, $n=5$ and $i=3$, hence we are talking about $2$, $2$, and $1$ strands, respectively.) One by one, these are isotopic to (i) a full twist on $i-1$ strands; (ii) a single crossing on two strands, represented by $\sigma_i$ in the word; (iii) a full twist on $n-i-1$ strands. These three smaller braids in turn are braided together to form $\delta\sigma_i$ in the fashion of $\Delta_3^2$. 

This time, we consider the pattern $\Delta_3^2$ as the braid word $\sigma_2\sigma_1\sigma_2\sigma_1\sigma_2\sigma_1=\,\includegraphics[width=.1\linewidth,viewport=2 1 25 10]{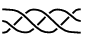}$. Note that its middle strand crosses under only twice, at the second and third crossings. Thus the part of the diagram that is thickened in the lower left of Figure \ref{fig:kiteker} contains no undercrossings. Hence it can be lifted and rearranged as shown in the lower right.
\end{proof}

\begin{lemma}\label{lem:newbranch}
Any ``extra branch'' of $\widetilde\Gamma$ created in our process (starting as shown in the lower part of Figure \ref{fig:kiteker}) does not contribute to the coefficient of $v^{n-1}$ in $H_{\widehat{\beta\Delta^2}}$.
\end{lemma}

\begin{proof}
We saw in Lemma \ref{lem:kiteker} that by using isotopy, conjugation, and a single Markov destabilization, $\alpha\delta\sigma_i$ can be reduced to a braid word $\eta\in B_{n-1}$. By \eqref{ineq:framed}, this implies that any exponent of $v$ in $H_{\widehat\eta}(v,z)$ is $n-2$ or less. The labels $v^{-1}$ that appear along the path connecting $\eta$ to the root $\beta\Delta^2$ can only reduce that exponent. 
\end{proof}

We shall now concentrate on the part of $\widetilde\Gamma$ that is isomorphic to $\Gamma$. Note that at the terminal nodes, where the trivial braids of $\Gamma$ used to be, now there are copies of $\Delta^2$.

\begin{lemma}\label{lem:oldnode}
The contribution of such a terminal node $\widetilde x$ of $\widetilde\Gamma$ to the coefficient of $v^{n-1}$ in $H_{\widehat{\beta\Delta^2}}$ agrees (up to sign) with the contribution of the corresponding terminal node $x$ of $\Gamma$ to the coefficient of $v^{1-n}$ in $H_{\widehat\beta}$. (In fact, both contributions are a single power of $z$.)
\end{lemma}

\begin{proof}
Let us assume that $x$ is labeled with 
$1_k$. To arrive (in $\Gamma$) from $\beta$ to $1_k$, there had to be exactly $n-k$ Markov destabilizations. Therefore, $h(x)=\pm z^{m}v^{k-n}$, where $m$ is a natural number. Thus, the coefficient of $v^{1-n}$ in \eqref{eq:contrib} becomes $\pm z^{m-k+1}$.

Now in $\widetilde\Gamma$, all $n-k$ of those $v^{-1}$ labels have been changed to $z$ (compare Figures \ref{fig:markov} and \ref{fig:kiteker}), while other labels along the path remained the same. Hence $h(\widetilde x)=\pm z^{m+n-k}$. By Proposition \ref{murakami}, we see that the coefficient of $v^{n-1}$ in the contribution of $\widetilde x$ is $\pm z^{m+n-k}\cdot(-1)^{n-1}z^{1-n}=\pm(-1)^{n-1}z^{m-k+1}$, as desired.
\end{proof}

Put together, Lemmas \ref{lem:newbranch} and \ref{lem:oldnode} conclude the proof of Proposition \ref{pro:framed}.

\begin{megj}
It is not hard to turn our argument into an induction proof of the MFW inequalities \eqref{ineq:framed}.
\end{megj}

\section{Remarks and examples}

There are several indications that coefficients in the left and right columns of the Homfly polynomial are more geometrically significant than others. For starters, these numbers persist under certain standard changes of variables and/or normalizing conditions. For example, if we require
$H'_{\includegraphics[totalheight=8pt]{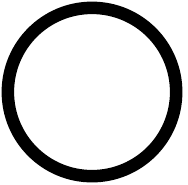}}=\frac{v^{-1}-v}z$ instead of $H_{\includegraphics[totalheight=8pt]{circle}}=1$,
then the left and right columns, up to sign and a downward shift, stay the same. The numbers (up to sign) do not change if we use $a=v^{-1}$, $l=-\sqrt{-1}\cdot v^{-1}$, $m=\sqrt{-1}\cdot z$ etc.\ either.

The next remark explains why the author examined these columns in the first place. The pattern of Theorem \ref{thm:jobbbal} emerged while using Knotscape \cite{knotscape} to browse through many examples.

\begin{megj}\label{rem:dan}
Rutherford \cite{dan} showed that if the knot type $K$ contains Legendrian representatives with sufficiently high Thurston--Bennequin number, then the coefficients in the left column of the Homfly polynomial $P_K(v,z)$ represent numbers of so-called oriented rulings (of various genera) of these Legendrian knots. Similarly, the right column may speak of oriented rulings of the mirror of $K$.

Theorems \ref{thm:jobbbal} and \ref{thm:preciz} however are not about oriented rulings. In particular, the method of the proof of Proposition \ref{murakami} does not extend to the general case. For instance, if $\beta$ is a positive braid (as in Example \ref{ex:torus}) then the left hand side of \eqref{eq:jobbbal} always contains counts of rulings \cite{imrn}, whereas the right hand side almost never does. 
\end{megj}

There is also the intriguing possibility that our results could be used to verify Jones's conjecture \cite{jones} that the writhe of a braid representative with minimum strand number is an oriented link invariant. Assume that the braids $\beta_1$ and $\beta_2$ are both on $n$ strands and they have isotopic closures, but their writhes disagree, say $w_1<w_2$. Suppose that the lower MFW estimate $w_2-n+1$ from $\beta_2$ is sharp. Then by \eqref{eq:jobbbal}, $P_{\widehat{\beta_2\Delta^2}}(v,z)$ contains non-zero coefficients at $v=w_2+n-1$, whereas by \eqref{eq:claim}, $P_{\widehat{\beta_1\Delta^2}}$ vanishes at $v=w_1+n-1$ and above. This is a contradiction if we also hypothesize that the closures of $\beta_1\Delta^2$ and $\beta_2\Delta^2$ are isotopic. There are too many ifs here, however. The last assumption certainly fails, as demonstrated in the next example. (Such counterexamples do not seem to be too common among small knots.)

\begin{pelda}\label{ex:12n187}
The $9$-crossing mirrored alternating knots $9_{27}^*$, $9_{30}^*$, and $9_{33}^*$ (Rolfsen's numbering) each have $4$-braid representatives that make the MFW inequality \eqref{ineq:mfw} sharp. For example, the braids listed in Knotinfo \cite{knotinfo}
\begin{multline*}
\sigma_1^{-2}\sigma_2\sigma_1^{-1}\sigma_2^{2}\sigma_3^{-1}\sigma_2\sigma_3^{-1}\text{, }\sigma_1^{-2}\sigma_2^{2}\sigma_1^{-1}\sigma_2\sigma_3^{-1}\sigma_2\sigma_3^{-1},\\
\text{and }\sigma_1^{-1}\sigma_2\sigma_1^{-1}\sigma_2^{2}\sigma_1^{-1}\sigma_3^{-1}\sigma_2\sigma_3^{-1},
\end{multline*} 
respectively have this property. Adding the full twist $\Delta_4^2$ to these braids, in each case we obtain a braid representative of the same knot $12n_{187}$ (from the Hoste-Thistlethwaite table). (In other words, $12n_{187}^*$ has three $4$-braid representatives so that adding a full positive twist to each, the three different knots $9_{27}$, $9_{30}$, and $9_{33}$ are obtained. Let us also mention that these braid representatives of $12n_{187}$ are mutually non-isotopic within the solid torus which is the complement of the braid axis.) 

The Homfly polynomials of $9_{27}^*$, $9_{30}^*$, and $9_{33}^*$ are different, but each has a left column (coefficient of $v^{-4}$) of $-z^2-1$. Of course, the right column (coefficient of $v^{14}$) in the Homfly polynomial of $12n_{187}$ is $z^2+1$.

$9_{27}$ is a slice knot while $9_{30}$ and $9_{33}$ are not. But even though they represent different concordance classes, all three have signature $0$. Their Khovanov and torsion Khovanov polynomials do not coincide yet have the same support with only minor differences in the coefficients.
\end{pelda}

To further underline that our main theorem is more about braids than knots, we end the paper with the following extension of Example \ref{ex:torus}. 

\begin{pelda}
We examine the effect of changing the braid $(\sigma_1\sigma_2)^5$ by positive and negative Markov stabilizations. Let us denote the resulting four-strand braids by $\beta$ and $\beta'$, respectively. Of course, the isotopy type of the closure of both is still the $(3,5)$ torus knot, but the lower MFW inequality \eqref{lower} remains sharp only in the first case. See Figure \ref{fig:mpm} for the Homfly polynomials after adding a full twist to each braid. Note how \eqref{eq:jobbbal} fails to work for $\beta'$. This is because in this case, \eqref{eq:claim} states that the columns of zeros indicated in Figure \ref{fig:mpm} agree. (In fact, the closure of $(\sigma_1\sigma_2)^5\sigma_3^{-1}\Delta_4^2$ is the torus knot $T(3,10)$.)

\begin{figure}
\labellist
\small
\pinlabel $z$ at 265 1550 
\pinlabel $z$ at 265 1070 
\pinlabel $z$ at 265 385 
\pinlabel $\beta\Delta^2=$ at 110 1355
\pinlabel $\beta'\Delta^2=$ at 120 190
\pinlabel $\beta'=$ at 80 875
\pinlabel $\beta=$ at 1010 875
\pinlabel $v=6$ at 550 680
\pinlabel $v=8$ at 640 1090
\pinlabel $v=26$ at 1465 1160
\pinlabel $v=24$ at 1375 -5
\pinlabel $v$ at 1575 1220
\pinlabel $v$ at 1575 740
\pinlabel $v$ at 1575 55
\pinlabel $\sigma_3^{-1}$ at 455 785
\pinlabel $\sigma_3$ at 1365 775
\endlabellist
\centering
\includegraphics[width=\linewidth]{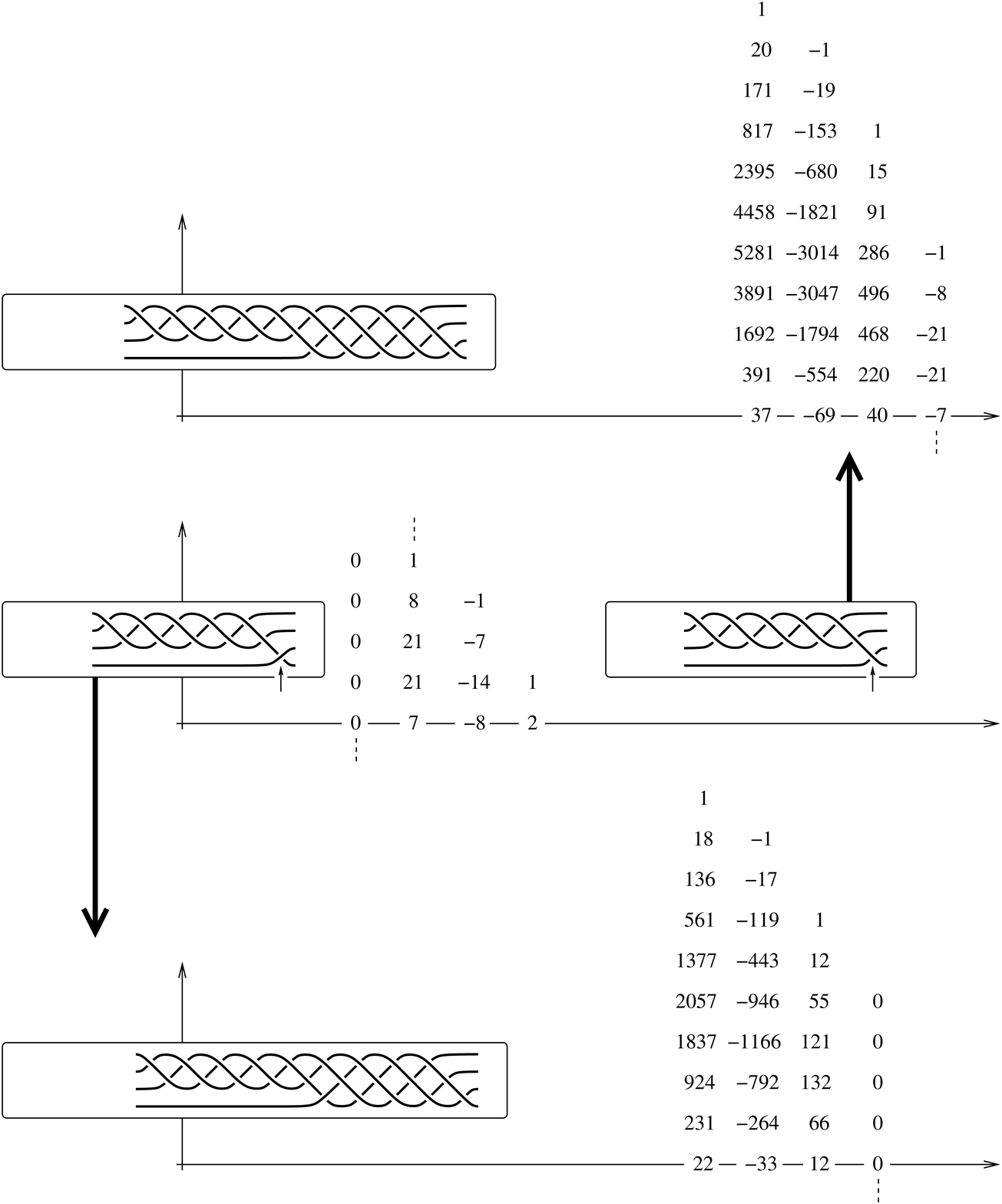}
\caption{Stabilizing a braid in two different ways before adding a full twist}
\label{fig:mpm}
\end{figure}

\end{pelda}

Acknowledgements: This paper was written while the author was a Japan Society for the Promotion of Science research fellow at the University of Tokyo. It is a particular pleasure to acknowledge the hospitality of Takashi Tsuboi. I would also like to thank Toshitake Kohno for his encouragement, as well as Lenny Ng for helpful e-mail discussions.



\bigskip

{\tt kalman@ms.u-tokyo.ac.jp}

\end{document}